\documentclass{amsart}
\usepackage[all]{xy}
\usepackage{mathrsfs}
\usepackage{amscd,amssymb,color, amsmath}

\let\sma\wedge
\newcommand{\htp}{\simeq}
\renewcommand{\to}{\mathchoice{\longrightarrow}{\rightarrow}{\rightarrow}{\rightarrow}}

\newcommand{\bsma}{\stackrel{B}{\sma}}

\DeclareFontFamily{OMS}{rsfs}{\skewchar\font'60}
\DeclareFontShape{OMS}{rsfs}{m}{n}{<-5>rsfs5 <5-7>rsfs7 <7->rsfs10 }{}
\DeclareSymbolFont{rsfs}{OMS}{rsfs}{m}{n}
\DeclareSymbolFontAlphabet{\scr}{rsfs}

\newcommand{\sI}{\scr{I}}
\newcommand{\sJ}{\scr{J}}

\newcommand{\cC}{{\mathcal C}}

\let\catsymbfont\mathcal

\newcommand{\aC}{{\catsymbfont{C}}}

\newcommand{\aI}{{\catsymbfont{I}}}

\newcommand{\aU}{{\catsymbfont{U}}}
\newcommand{\aV}{{\catsymbfont{V}}}

\newcommand{\D}{\mathcal{D}}

\newcommand{\I}{\mathcal{I}}
\newcommand{\J}{\mathcal{J}}

\newcommand{\bJ}{\mathbb{J}}

\newcommand{\bN}{{\mathbb{N}}}

\newcommand{\bR}{{\mathbb{R}}}

\newcommand{\bn}{\mathbf{n}}

\newcommand{\TC}{\textnormal{TC}}
\newcommand{\THH}{\textnormal{THH}}
\newcommand{\HH}{\textnormal{HH}}

\DeclareMathOperator*{\hocolim}{hocolim}

\def\quickop#1{\expandafter\DeclareMathOperator\csname
#1\endcsname{#1}}
\quickop{id}\quickop{Id}\quickop{colim}\quickop{op}
\quickop{co}\quickop{Ar}\quickop{sing}\quickop{Hom}\quickop{w}\quickop{Ho}
\quickop{ob}\quickop{diag}\quickop{Stab}\quickop{Cat}\quickop{Mot}\quickop{Mod}
\quickop{map}\quickop{Cone}\quickop{End}\quickop{Idem}\quickop{Perf}\quickop{Ind}
\quickop{Gap}\quickop{Shv}\quickop{Spaces}\quickop{cof}\quickop{ex}\quickop{perf}
\quickop{tri}\quickop{Set}\quickop{Fib}\quickop{Nat}\quickop{haut}\quickop{holim}\quickop{Tor}\quickop{Ext}\quickop{Diff}\quickop{lax}\quickop{Map}\quickop{Comm}\quickop{Top}\quickop{Tr}\quickop{Res}\quickop{conn}\quickop{cyc}

\numberwithin{equation}{section}

\newtheorem{theorem}[equation]{Theorem}
\newtheorem*{theorem*}{Theorem}
\newtheorem{corollary}[equation]{Corollary}
\newtheorem{lemma}[equation]{Lemma}
\newtheorem{proposition}[equation]{Proposition}
\theoremstyle{definition}
\newtheorem{definition}[equation]{Definition}
\newtheorem{remark}[equation]{Remark}

\xyoption{arrow}
\xyoption{matrix}
\xyoption{cmtip}
\SelectTips{cm}{}

\newdir{ >}{{}*!/-5pt/\dir{>}}

\begin{document}

\title{Interpreting the B\"okstedt smash product as the norm}

\author[V. Angeltveit]{Vigleik Angeltveit}
\address{Australian National University, Canberra, ACT 0200, Australia}
\email{vigleik.angeltveit@anu.edu.au}
\thanks{V.~Angeltveit was supported in part by an NSF All-Institutes Postdoctoral Fellowship administered by the Mathematical Sciences Research Institute through its core grant DMS-0441170, NSF grant DMS-0805917, and an Australian Research Council Discovery Grant}
\author[A.J. Blumberg]{Andrew J. Blumberg}
\address{Department of Mathematics, University of Texas,
Austin, TX \ 78712}
\email{blumberg@math.utexas.edu}
\thanks{A.~J.~Blumberg was supported in part by NSF grant DMS-0906105}
\author[T. Gerhardt]{Teena Gerhardt}
\thanks{T.~Gerhardt was supported in part by NSF DMS--1007083 and NSF DMS--1149408}
\address{Michigan State University, East Lansing, MI 48824}
\email{teena@math.msu.edu}
\author[M.A.Hill]{Michael A.~Hill}
\address{University of Virginia \\ Charlottesville, VA 22904}
\email{mikehill@virginia.edu}
\thanks{M.~A.~Hill was supported in part by NSF DMS--0906285, DARPA FA9550-07-1-0555, and the Sloan Foundation}
\author[T. Lawson]{Tyler Lawson}
\address{University of Minnesota \\ Minneapolis, MN 55455}
\email{tlawson@math.umn.edu}
\thanks{T.~Lawson was supported in part by NSF DMS--1206008 and the Sloan Foundation.}


\begin{abstract}
This paper compares two models of the equivariant homotopy type of the
smash powers of a spectrum, namely the ``B\"okstedt smash product''
and the Hill-Hopkins-Ravenel norm.
\end{abstract}

\maketitle
\setcounter{tocdepth}{1}

\section{Introduction}\label{sec:intro}

In any symmetric monoidal category $\aC$, a basic construction of a
$C_n$-equivariant object arises from the smash power 
\[
X \mapsto X^{\otimes n} = \underbrace{X \otimes X \otimes \ldots \otimes X}_{n}.
\]
The homotopical analysis of such constructions in the category of
spaces and spectra is of classical importance.  For instance, Steenrod
operations arise from the homotopy coinvariants $C^*(X^{\times
n})_{h\Sigma_n}$ of $X^{\times n}$ with respect to the $\Sigma_n$-action.  More general power operations arise from the analysis of this
kind of construction in the category of spectra.

Our focus in this paper is the analysis of the smash power
construction in the equivariant stable category.  In contrast to the
setting of spaces, there are two distinct fixed-point functors in the
category of equivariant spectra: for a $G$-equivariant $X$ and
subgroup $H \subset G$, we can construct ``categorical fixed points'' $X^H$
and ``geometric fixed points'' $\Phi^{H} X$.  The
interplay and contrast between these two functors encodes much of the
complexity of the equivariant stable category.

An important application of the smash power construction of spectra
arises in the study of trace methods for computing algebraic
$K$-theory.  Algebraic $K$-theory has been revolutionized in the last
20 years by the development of trace methods.  Following ideas of
Goodwillie~\cite{Goodwillie}, first B\"okstedt~\cite{Bokstedt} and then
B\"okstedt, Hsiang, and Madsen~\cite{BHM} constructed topological
analogues of Hochschild homology and negative cyclic homology, called topological Hochschild homology, $\THH(R)$, and topological cyclic homology, $\TC(R)$. They also constructed a ``cyclotomic'' trace map 
\[
K(R) \to \TC(R) \to \THH(R)
\]
lifting the classical Dennis trace $K(R) \to \HH(R)$.  The
relationship between relative $K$-theory and relative $\TC$ is well
understood (after $p$-completion) in many
circumstances~\cite{McCarthy, Dundas}.  The target of the trace,
$\TC(R)$, is constructed from an equivariant structure that arises on
$\THH(R)$; exploiting the tools of equivariant homotopy theory makes
$\TC(R)$ relatively computable.

The equivariant structure on $\THH(R)$ arises from the interpretation
of (topological) Hochschild homology as the cyclic bar construction.
In the category of spaces, for a group-like topological monoid $M$
there is an equivalence $N^{\cyc} M \htp LBM$, where for a space $X$, $LX = \Map(S^1,X)$
is the free (unbased) loop space on $X$.  The $S^1$-action on $LX$ has
an unusual property.  For a finite subgroup $H \subset S^1$,
$S^1/H \cong S^1$, and pulling back along this isomorphism induces a
homeomorphism of $S^1$-spaces 
\[
\Map(S^1,X)^H \cong \Map(S^1, X).
\]
In the category of $S^1$-spectra, the analogue of this structure is
called a cyclotomic structure.  In this context, $\THH$ being a cyclotomic
spectrum boils down to having (suitably coherent) ``diagonal''
equivalences of geometric fixed points
\[
\Phi^{C_n}(X^{\sma n}) \htp X.
\]  

B\"okstedt introduced $\THH(R)$ before the invention of symmetric
monoidal categories of spectra; he invented coherence machinery
(which anticipated the development of symmetric spectra, as explained
in~\cite{shipley}) to handle the smash product and proved that his
model of smash powers had the right homotopy type.
After the invention of modern categories of spectra, it became
possible to give definitions of $\THH(R)$ that simply computed the cyclic
bar construction in the usual way, circumventing the complexity of
B\"okstedt's coherence machinery.  However, it was believed that the
smash-power in this context did not have the right fixed points,
and so a direct construction of $\THH(R)$ as an equivariant spectrum
using the modern categories of spectra was thought to be out of reach (e.g., 
see~\cite[2.5.9]{Madsen} and~\cite[\S IX.3.9]{EKMM}).

Very recently, the solution to the Kervaire invariant one problem by
the fourth author, Hopkins, and Ravenel involved development of a
``norm'' functor $N_H^G$ from $H$-spectra to $G$-spectra which has the
correct diagonal fixed points~\cite{HHR}.  In particular, there is an
equivalence
\[
R \to \Phi^G N_e^{G} R
\]
for any finite group $G$ and cofibrant $R$.  When $G = C_k$, the
underlying spectrum of $N_e^{C_k}R$ is precisely the smash power
$R^{\sma k}$.

This behavior strongly suggests that the norm should agree with
B\"okstedt's version of the smash powers.  The purpose of this paper
is to make this precise, by constructing an explicit comparison
between the two as equivariant spectra.  We build on earlier
analysis by Shipley~\cite{shipley} which interpreted the B\"okstedt
construction in terms of a ``detection functor'' in symmetric spectra,
and is somewhat related to analysis done by Lun{\o}e-Nielsen and
Rognes~\cite{LNR}.  

In the sequel to this paper, we use this
comparison to show that the norm construction $N_e^{S^1} R$ for a
cofibrant ring spectrum $R$ directly yields a model of $\THH(R)$ as
a cyclotomic spectrum.  This new model of $\THH$ in terms of the norm
allows us to define new ``relative'' versions of $\TC$; in one, we
work with the smash product over a commutative ring spectrum $A$, and
in the other, we start with ring orthogonal $C_n$-spectra as input.
The computational machinery developed in~\cite{HHR} then can be
applied to provide new approaches for computation, ultimately leading
to calculations in algebraic $K$-theory.

To state our main results, we begin by fixing some notation and
definitions for our indexing categories. Let $\I$ denote B\"okstedt's
indexing category, i.e., the category with objects $\textbf{n}
= \{1,2,\ldots,n\}$ and morphisms all injections. Let $\J$ denote the
subcategory of $\I$ with the same objects but maps the subset
inclusions. Notice that $\J$ has a unique map from $\bf m$ to $\bf n$
for $m<n$.  

Recall the following definition of the B\"okstedt smash product:
\begin{definition}
Let  $X(1), X(2), \ldots, X(k)$ be symmetric spectra in spaces.
We define the orthogonal B\"okstedt smash product to the be orthogonal spectrum with $W$\textsuperscript{th} space 
\[
\big(X(1) \bsma \ldots \bsma X(k)\big)_{W} = \hocolim_{(\bf n_1,\ldots, \bf
n_k) \in \I^k} \Omega^{n_1+\ldots+n_k}\big(X(1)_{n_1} \sma \ldots \sma X(k)_{n_k} \sma S^W\big). 
\]
\end{definition}

Note that more generally we get a continuous functor from finite
based $CW$-complexes to spaces by plugging in a space $A$ in place of
$S^W$; by restriction, we can extract an orthogonal spectrum or
symmetric spectrum.

Fix a complete $C_k$-universe $\aU$.  For our model of the equivariant
stable category, we use the category of orthogonal
$C_k$-spectra~\cite{MM}.  Specializing to the case of the smash-power of
a single spectrum $X$, observe that $X^{\bsma k}$ becomes a
$C_k$-equivariant orthogonal spectrum indexed on $\aU$ as we let $W$
vary through the finite-dimensional $C_k$-inner product spaces that embed in $\aU$.  Here $C_k$
acts by conjugation on the mapping space, permutation on the indexing
category, and the action on $S^W$.

The following is the main theorem of this paper.

\begin{theorem}\label{Main}
Let $X$ be a cofibrant orthogonal spectrum, and let $\tilde{X}$ denote
the underlying symmetric spectrum (of topological spaces).  Then there is an isomorphism in the
homotopy category of $C_k$-equivariant orthogonal spectra
\[
\tilde{X}^{\bsma k} \cong N_e^{C_k} X,
\]
where $N_e^{C_k} X$ is the Hill-Hopkins-Ravenel norm.

Moreover, this isomorphism is induced by a zig-zag of maps natural in
$X$.
\end{theorem}

To explain the proof, we introduce some further notation. Let $\bJ$ denote the category of finite dimensional real inner product spaces $V$ in $\bR^\infty$, with morphisms the inclusions $V \to W$ in $\bR^\infty$. We regard this as a discrete category. We will denote by $N_e^{C_k}$ both the norm from spectra to $C_k$-spectra as well as the norm from spaces to $C_k$-spaces; the usage will be clear from context.


Let $\rho$ denote the regular representation of $C_k$. We now have the
following main comparison diagram in the category of $C_k$-spaces: 
\begin{equation}\label{eq:main}
\xymatrix@R=2em{
(\tilde{X}^{\bsma k})_W \ar[d]^{=} \\
\hocolim\limits_{\I^k} \Omega^{n_1 + \ldots +
n_k}(\tilde{X}_{n_1} \sma \ldots \sma \tilde{X}_{n_k} \sma S^W) \\
\hocolim\limits_{\I} \Omega^{kn}(\tilde{X}_{n} \sma \ldots \sma \tilde{X}_{n} \sma S^W) = \hocolim\limits_{\I} \Omega^{kn}(N_e^{C_k} (\tilde{X}_{n}) \sma S^W) \ar[u]_{\Delta} 
\\
\hocolim\limits_{\J} \Omega^{kn}(N_e^{C_k} (\tilde{X}_{n}) \sma S^W) \ar[u]_{D_1} \ar[d]^{D_2}
\\
\hocolim\limits_{\bJ} \Omega^{V \otimes \rho}(N_e^{C_k}(X_{V}) \sma
S^W)
}
\end{equation}
The map labeled $\Delta$ is the diagonal inclusion.  The maps
labeled $D_i$ are induced from the natural inclusion $\J \to \I$ and
the functor $\J \to \bJ$ given (on objects) by ${\bf
m} \mapsto \bR^m$.

The proof of the main theorem amounts to showing that all of the
vertical maps assemble to weak equivalences of equivariant orthogonal
spectra.  There are essentially three parts of the argument:
\begin{enumerate}
\item establishing a stable equivalence of equivariant orthogonal spectra
\[
\left\{ W \mapsto \hocolim_{\bJ} \Omega^{V \otimes \rho}(N_e^{C_k}(X_{V}) \sma
S^W) \right\} \htp (N_e^{C_k} X)
\] 
(Section~\ref{sec:canonical}),
\item  establishing the comparisons associated to changing the indexing
category (Section~\ref{sec:diagrams}), and
\item studying the diagonal map $\Delta$ (Section~\ref{sec:diagonal}).
\end{enumerate}

\bigskip

The authors would like to thank the American Institute of Mathematics (AIM)
for its support through the SQuaREs program and MSRI for its
hospitality while some of this work was being done.  The authors would
like to thank Mike Mandell for his assistance, as well as Lars
Hesselholt and Mike Hopkins for interesting and useful conversations.

\section{A brief review of equivariant orthogonal spectra and the norm}

We begin with a very rapid review of the details of the theory of
equivariant orthogonal spectra that we need.  The authoritative source
on this subject is~\cite{MM}; we also require a number of the
refinements of the foundations developed in~\cite{HHR}.  See
also~\cite[\S 2]{BM} for a concise review.

\subsection{Basic definitions}
\label{sec:definitions}

Fix a finite group $G$.  A $G$-universe $\aU$ is a $G$-inner product
space isomorphic to $\bigoplus_{n \in \bN} V$ for some finite-dimensional
representation $V$ which contains a trivial
representation~\cite[II.1.1]{MM}.  We say $\aU$ is complete when it
contains all irreducible representations of $G$.  Denote by $\aV=\aV(\aU)$
the collection of all finite-dimensional $G$-inner product spaces that
are isomorphic to sub-$G$-inner product spaces of
$\aU$~\cite[II.2.1]{MM}.

Given two representations $V,W \in \aV$, we define $\sI_G(V,W)$ to
be the $G$-space of non-equivariant isometric isomorphisms $V \to W$.
This construction allows us to regard $\aV$ as the objects of a
category $\sI_G$ enriched in $G$-spaces.  An orthogonal $G$-spectrum
is now defined to be an enriched functor $X$ from $\sI_G$ to the
category of $G$-spaces equipped with an equivariant natural
transformation 
\[
X_V \sma S^W \to X_{V \oplus W}
\]
that satisfies the obvious transitivity property~\cite[\S II.2]{MM}.

\begin{remark}
A convenient aspect of this definition is that any continuous (i.e.,
enriched) functor from $G$-spaces of the homeomorphism type of finite
$G$-CW complexes to $G$-spaces gives rise to an orthogonal
$G$-spectrum by restriction of the domain.  This observation is
relevant to our work because many of the orthogonal $G$-spectra we
study arise in this fashion.
\end{remark}

The category of orthogonal $G$-spectra has a symmetric monoidal
structure induced by the Day convolution.  In order to describe this,
we recall an alternate description of the category of orthogonal
$G$-spectra~\cite[\S II.4]{MM}.  For representations $V$ and $W$, let
$\sJ_G(V,W)$ denote the Thom space of the orthogonal complement
$G$-bundle of linear isometries $V \to W$.  We now have a category
$\sJ_G$ with the same objects as $\sI_G$, and an orthogonal
$G$-spectrum is precisely an enriched functor from $\sJ_G$ to
$G$-spaces.  Now the Day convolution defines the symmetric monoidal
structure, with unit the sphere spectrum $S$.

This description of orthogonal $G$-spectra is also useful because it
provides an explicit description of the adjoint of the evaluation
functors.  For a $G$-representation $V$ and a based $G$-space $X$, we
define the shift-desuspension spectrum $S^{-V}\wedge X$ to have $W$'th space
$\sJ_G(V,W) \sma X$~\cite[II.4.6]{MM}.  Of particular importance are
the negative $V$-spheres $S^{-V}$.

We now turn to the homotopy theory of orthogonal $G$-spectra.  The
category of orthogonal $G$-spectra admits a model structure such that
the homotopy category is equivalent to the equivariant stable
category.  To describe this model structure, we need to review the
notion of equivariant stable homotopy groups.

The homotopy groups of an orthogonal $G$-spectrum $X$ are defined for
a subgroup $H \subset G$ and an integer $q$ by
\[
\pi_q^H(X) =
\begin{cases}
\quad\displaystyle 
\colim_{V \in \bJ} \pi_{q}\big((\Omega^{V}X_V)^{H}\big)&\text{if }q\geq 0,\\[1em]
\quad\displaystyle 
\colim_{\bR^{-q} \subset V \in \bJ} \pi_0\big((\Omega^{V-\bR^{-q}}X_V)^{H}\big)&\text{if }q< 0\\
\end{cases}
\]
(see~\cite[\S III.3.2]{MM}).  We define the stable equivalences to be
the maps $X \to Y$ that induce isomorphisms for all homotopy groups.
There is a cofibrantly-generated monoidal model structure on
orthogonal $G$-spectra in which the weak equivalences are the stable
equivalences~\cite[\S IV.2]{MM}.

The stable equivalences are chosen so that the canonical map
\[
S^{-V}\wedge S^{V}\to S^{0}
\]
is a stable equivalence for all $V$~\cite[III.4.5]{MM}. In particular,
this means that for any $G$-space $X$, we have a stable equivalence 
\begin{equation}\label{eqn:StabEquiv}
S^{-V}\wedge S^{V}\wedge X\to \Sigma^{\infty} X.
\end{equation}

We now recall an elementary properties of shifts of orthogonal
spectra.  For an orthogonal $G$-spectrum $X$, recall that
$\Omega^{V}X$ denotes the orthogonal $G$-spectrum with 
\[
(\Omega^{V}X)_{W}=\Omega^{V}(X_{W}).
\]

\begin{proposition}\label{prop:LoopsDesuspension}
If $X$ is a $G$-space, then we have a natural stable equivalence 
\[
S^{-V}\wedge X\to \Omega^{V}\Sigma^{\infty}X
\]
of orthogonal $G$-spectra.
\end{proposition}

\begin{proof}
The map in question is the adjoint of the stable equivalence given by
Equation~\ref{eqn:StabEquiv}: 
\[
S^{-V}\wedge X\to \Omega^{V}\Sigma^{\infty}X.
\]
By~\cite[III.3.6]{MM}, a map $f$ of orthogonal $G$-spectra is a stable
equivalence if and only if $\Sigma^V f$ is a stable equivalence.
Therefore, it suffices to check that 
\[
S^V \sma S^{-V}\wedge X\to \Sigma^V \Omega^{V}\Sigma^{\infty}X.
\]
But using Equation~\ref{eqn:StabEquiv} again, we can reduce to the
question of whether the unit
\[
\Sigma^\infty X \to \Sigma^V \Omega^V \Sigma^{\infty} X
\]
is a weak equivalence, which follows by~\cite[III.3.8]{MM}.
\end{proof}

\subsection{Canonical Homotopy Presentations}
\label{sec:canonical}
Up to stable equivalence, there is a convenient way of expressing
the homotopy type of an orthogonal $G$-spectrum in terms of its
spaces; this is referred to as the ``canonical homotopy presentation''
in~\cite[B.1.5]{HHR}.  Specifically, let 
\[
\ldots \subset V_n \subset V_{n+1} \subset \ldots
\]
be an exhausting sequence of orthogonal $G$-representations (by which
we mean that every finite dimensional $G$-representation $V$ in
$\aV(\aU)$ can be embedded in some $V_n$) and write $W_n$ for the
orthogonal complement of $V_n$ in $V_{n+1}$.  We obtain zig-zags of
the form 
\[
\xymatrix{
S^{-V_n} \sma X_{V_n} & \ar[l]_-{\htp} S^{-(V_n \oplus W_n)}\sma S^{W_n} \sma
X_{V_n} \ar[r] & S^{-V_{n+1}} \sma X_{V_{n+1}} \\
}
\] 
and taking the homotopy colimit over the first $m$ stages we get
something equivalent to $S^{-V_m} \sma X_{V_m}$.  Therefore, we write
\[
\hocolim_{V_n} S^{-V_n} \sma X_{V_n}
\]
to denote the homotopy colimit over this zig-zag.  The result is
stably equivalent to $X$.

Combining this with Proposition~\ref{prop:LoopsDesuspension} gives a convenient description of orthogonal spectra.

\begin{lemma}\label{lem:ReCanonical}
If $X$ is an orthogonal $G$-spectrum and $\{V_{n}\}$ is an exhausting
sequence of orthogonal $G$-representations, then we have a stable
equivalence  
\[
X\simeq \Big(W\mapsto \hocolim_{V_{n}} \Omega^{V_{n}}\big(X_{V_{n}}\wedge S^{W}\big) \Big).
\]
\end{lemma}

The exhausting sequences also provide a way to prove that the map $D_2$ of (\ref{eq:main}) is an equivalence. Any choice of exhausting sequence $V_{n}$ in $\mathbb R^{\infty}$ defines a functor $\J\to\bJ$ by
\[
\bn\mapsto V_{n}.
\]
Stated in this way, the sequence being exhausting is the same thing as the homotopy cofinality of this functor  (e.g., see \cite[A.5]{Lind}). 

\begin{lemma}
Let $X$ be an orthogonal $G$ spectrum.  For any exhausting sequence $V_{n}$ in $\mathbb R^{\infty}$, the map
\[
\xymatrix{
\hocolim\limits_{n\in \J} \Omega^{V_{n}\otimes\rho}(X_{V_{n}\otimes\rho} \sma S^W) \ar[r]^-{D_2}
&
\hocolim\limits_{V\in\bJ} \Omega^{V\otimes\rho}(X_{V\otimes\rho} \sma
S^W) \\
}
\]
is an equivariant equivalence.

In particular, the map $D_{2}$ is always an equivalence.
\end{lemma}

\begin{proof}
The map $D_2$ is an equivariant equivalence because the inclusion
$\J \to \bJ$ is homotopy cofinal and passage to fixed points commutes with filtered homotopy colimits. 
\end{proof}

\subsection{The Hill-Hopkins-Ravenel norm}

We now review the special case of the norm construction that we need.
Given a non-equivariant orthogonal spectrum $X$, the smash power
$X^{\sma n}$ can be given a $C_n$-action and regarded as an orthogonal
$C_n$-spectrum indexed on the universe containing only trivial
representations (which we denote by $\bR^{\infty}$).  In order to
obtain an orthogonal $C_n$-spectrum indexed on the complete
$C_n$-universe $\aU$, we simply apply the point-set change-of-universe
functor $\aI_{\bR^{\infty}}^{U}$~\cite[\S V.1]{MM}.  This composite
defines the norm as
\[
N_e^{C_n} X = \aI_{\bR^{\infty}}^{U} X^{\sma n},
\]
and more generally the analogous construction can be used to define
$N_e^G X$ for any finite group $G$.  There is a more complicated
general construction of the norm $N_H^G X$ which takes orthogonal $H$-spectra
to orthogonal $G$-spectra for arbitrary subgroups $H \subset G$, but
since we do not need that herein we refer the interested reader
to~\cite[\S A.4]{HHR} for details.

A key technical insight of Hill-Hopkins-Ravenel is that the norm
$N_e^G$ preserves weak equivalences between cofibrant orthogonal
$G$-spectra and, in fact, cofibrant (commutative) ring orthogonal
$G$-spectra~\cite[\S B.2]{HHR}.  That is, cofibrant replacement allows us
to compute the derived functor of the norm.

We need one further fact about the norm, which relates the norm to a
description in terms of the canonical homotopy presentation and the
norm functor on spaces~\cite[\S 2.3.2]{HHR}.  For a space $A$, define 
\[ 
N_e^G(A) = \bigwedge_G A,
\]
regarded as a $G$-space via the permutation action.
Then we have the following description of the norm:
\[ N_e^G(X) = \hocolim_{V_n} S^{-V_n \otimes \rho} \sma
N_e^G(X_{V_n}), \]
where $\rho$ is the regular representation and $V_n$ is an increasing sequence of subspaces of $\bR^\infty$.  Note that as above, the
hocolim notation hides the fact that the diagram has backwards weak
equivalences.

Since for any exhausting sequence $V_{n}$ for $\mathbb R^{\infty}$, the sequence $V_{n}\otimes\rho$ is exhausting for a complete universe, the canonical homotopy presentation allows us another description of the norm.
\begin{corollary}
If $X$ is a cofibrant orthogonal spectrum, then there is a stable equivalence of orthogonal $G$-spectra
\[
N_{e}^{G}(X)\simeq\Big(W\mapsto\hocolim_{V\in\bJ} \Omega^{V\otimes\rho}\big(N_{e}^{G}(X_{V})\wedge S^{W}\big)\Big).
\]
\end{corollary}

\section{Changing diagrams}\label{sec:diagrams}

In this section, we show that the ``change of indexing diagram'' maps
are equivariant stable equivalences.  

To study $D_1$, we need to establish an equivariant version of
B\"okstedt's telescope lemma comparing homotopy colimits over $\I$ to
homotopy colimits over $\J$ (i.e., telescopes).
For this, we need to recall that in equivariant homotopy, there is a
refined notion of ``connectedness'' which records the behavior of
fixed points for subgroups (as opposed to the coarser notion of
``inducing an isomorphism of homotopy groups through a
range''). Let $\nu$ be a function from conjugacy classes of
subgroups of $G$ to $\mathbb N\cup\{\infty\}$. A $G$-equivariant map
$f\colon X\to Y$ is said to be ``$\nu$-connected'' if
\[
f^{H}\colon X^{H}\to Y^{H}
\] 
is $\nu(H)$-connected for all subgroups $H \subset G$. With this
language, we can give an equivariant refinement of the telescope
lemma.   We use the formulation of the telescope lemma due to
Schlichtkrull~\cite[2.2]{Schlichtkrull}.  In the following, we will
refer to a functor $X\colon \I\to \Top_{G}$ as an $\I-G$-space.  

\begin{lemma}\label{lem:Telescope}
Let $X$ be an $\I-G$-space and suppose that each morphism $\bf n_{1}\to
\bf n_{2}$ in $\I$ with $n_{1}\geq n$ induces a $\nu_{n}$-connected map 
\[
X(n_{1})\to X(n_{2}).
\]
Then given any $m\geq n$, the natural map 
\[
X(m)\to \hocolim_{\I} X(i)
\]
is at least $(\nu_{n}-1)$-connected.

In particular, if for each $H$, $\nu_{n}(H)$ goes to infinity, then 
\[
\hocolim_{\J} X(i)\to \hocolim_{\I} X(i)
\]
is an equivariant equivalence.
\end{lemma}

\begin{proof}
Since the group $G$ does not act on $\I$, we know that the fixed
points commute with the homotopy colimit: there is a natural
equivalence 
\[
\hocolim_{\I}(X(i)^{H})\xrightarrow{\simeq} \big(\hocolim_{\I} X(i)\big)^{H},
\]
where the map is induced from the natural inclusion $X(i)^{H}\to X(i)$.

By assumption, for all $\bf n_{1}\to \bf n_{2}$ and $n_{1}\geq n$ and for all
$H$, we have 
\[
X(n_{1})^{H}\to X(n_{2})^{H}
\]
is $\nu_{n}(H)$-connected. By the usual telescope
lemma~\cite[2.2]{Schlichtkrull}, we conclude that 
\[
X(m)^{H}\to\hocolim_{\I} X(i)^{H}\xrightarrow{\simeq}\big(\hocolim_{\I} X(i)\big)^{H}
\]
is $(\nu_{n}(H)-1)$-connected. 

The proof of the second part is immediate from the first, since by
elementary cofinality arguments in $\J$, we see that under the given
assumptions, the map 
\[
\hocolim_{\J} X(j)\to\hocolim_{\I} X(i)
\]
is infinitely equivariantly connected.
\end{proof}

We will need some simple observations about the connectivity of a
norm in spaces or of the norm of a map between spaces. 

\begin{lemma}\label{lem:normconnectivity}
The equivariant connectivity of $N_e^G(X)$ is given by 
\[
\nu_{X}(H)=\tfrac{|G|}{|H|}(\conn(X)+1)-1\geq \tfrac{|G|}{|H|}\conn(X),
\]
while for a map $f\colon X\to Y$,
\[
\nu_{f}(H)=\conn(f)+(\tfrac{|G|}{|H|}-1)(\conn(X)+1).
\]
\end{lemma}

\begin{proof}
Both results rely on the diagonal: if $Y$ has the trivial action, the map 
\[
Y^{|G|/|H|}=N_{H}^{G}(Y)\to N_{e}^{G}(Y)
\] 
is a homeomorphism on $H$-fixed points for any normal subgroup
$H$. Standard arguments about the connectivity of non-equivariant
smash powers of a map then show that the equivariant connectivity of
$N_{e}^{G}(X)$ and $N_{e}^{G}(f)$ is given by the listed formulas.
(See also~\cite[3.12]{BHM} for discussion of this phenomenon.)
\end{proof}

Similarly, we recall a small result about the connectivity of an equivariant mapping space (e.g., see~\cite[2.5]{HeMa97}).

\begin{lemma}\label{lem:mappingspaceconnectivity}
Suppose $X$ and $Y$ are $G$-spaces with $X$ having the structure of an $H$-CW complex.  Then
\begin{equation}\label{eq:conn}
\conn(\Map(X,Y)^{H}) \geq \min_{K \subset H} \big( \conn(Y^K) - \dim(X^K) \big).
\end{equation}
\end{lemma}
This will allow us to easily get very coarse lower bounds.

\begin{proposition}\label{prop:freudenthal}
Let $G = C_k$. For a space $X$, the suspension map
\[
E^{m\rho}\colon N_{e}^{G} X \to \Omega^{m\rho} \Sigma^{m\rho} N_{e}^{G}X
\]
has connectivity at least
\[
\nu(H)=2\tfrac{k}{|H|}\conn(X).
\]
\end{proposition}

\begin{proof}
We use the equivariant Freudenthal suspension theorem\footnote{For the interested reader, the first condition following is perhaps what one would expect: the connectivity is exactly what the usual Freudenthal suspension theorem would predict, just for each of the fixed points. The second condition is somewhat less obvious: this condition guarantees that the connectivity of $H$-equivariant maps between two spaces is the same as the ordinary maps between their $H$-fixed points.}. Recall that if $V$ is a representation of $G$, then the suspension map
\[
E^{V}\colon Y\to \Omega^{V}\Sigma^{V}Y
\]
is $\nu_{V}$-connected for any $\nu_{V}$ satisfying
\begin{enumerate}
\item For all $H$ such that $V^{H}\neq\{0\}$, we have $\nu_{V}(H)\leq 2\conn(Y^{H})+1$.
\item For all $K\subset H$ with $V^{K}\neq V^{H}$, we have $\nu_{V}(H)\leq\conn(Y^{K})$.
\end{enumerate}

We now consider $V=m\rho$ and $Y=N_{e}^{G}(X)$. For all subgroups $H$ of $G$, $V^{H}\neq \{0\}$, and if $K\subsetneq H$, then $V^{K}\neq V^{H}$. If we let $\nu$ be as in the proposition statement, then the essential condition to check is the second one. So we must show that for all $K\subsetneq H$,
\[
\nu(H)=2\tfrac{k}{|H|}\conn(X)\leq\conn(Y^{K})=\tfrac{k}{|K|}\big(\conn(X)+1\big)-1.
\]
However, since $K\neq H$, we know that this holds. Thus by the
equivariant Freudenthal suspension theorem, the map
$E^{m\rho}$ is $\nu$-connected. 
\end{proof}

To prove that the map $D_{1}$ is an equivariant equivalence, we will
directly verify the needed connectivity hypotheses to apply
Lemma~\ref{lem:Telescope}.  


\begin{theorem}
Let $X$ be a cofibrant orthogonal spectrum with underlying symmetric spectrum
$\tilde X$.  Then the maps 
\[
\xymatrix{
\hocolim\limits_{\J} \Omega^{kn}(N_e^{C_k}\big(\tilde X_{n}) \sma S^W\big)
\ar[r]^-{D_1} & \hocolim\limits_{\I} \Omega^{kn}\big(N_e^{C_k}(\tilde X_{n}) \sma S^W\big) \\
}
\]
are equivalences of $C_k$-spaces and therefore assemble to a stable
equivalence of orthogonal $C_k$-spectra.
\end{theorem}

\begin{proof}
The first part of the conclusion implies the second; since we show
that the map $D_1$ establishes level equivalences of equivariant
orthogonal spectra, it gives an equivariant stable equivalence. 

Since $\tilde X$ is the restriction of an orthogonal spectrum, it is a semistable symmetric spectrum \cite{shipley}. Thus Shipley's version of B\"okstedt's Telescope Lemma (\cite[Corollary 3.1.7]{shipley}) shows that $D_{1}$ induces an underlying equivalence.

Since both sides of the map preserve stable equivalences of orthogonal $G$-spectra (i.e.,
see~\cite[2.1.9]{shipley} for the righthand side), we can simplify our
analysis using the canonical homotopy presentation. Both the domain and range commute with filtered homotopy colimits in $X$ because $S^{n\rho}$ is compact, and the canonical homotopy
presentation allows us to express $X$ as a filtered homotopy colimit
of objects of the form $S^{-V}\wedge Y$.  Therefore, it suffices to prove this
result for orthogonal spectra of the form $X = S^{-V}\wedge Y$ with $Y$
cofibrant and $V$ a $d$-dimensional vector space.  In this case, the
underlying symmetric spectrum has $\tilde X_n = \sJ(V,\bR^n) \wedge
Y$.

To show that $D_1$ induces an
equivariant equivalence, we verify that the $\I-C_k$-space 
\[
\bn\mapsto \Omega^{n\rho}\big(N_{e}^{C_{k}}(\tilde X_{n})\wedge S^{W}\big)
\]
has the appropriate connectivity properties for
Lemma~\ref{lem:Telescope}. Here we have replaced the more traditional
$\Omega^{kn}$ with the visibly equal functor $\Omega^{n\rho}$ to
strengthen the equivariant connection in the reader's mind. 

The presence of $S^{W}$ does not affect the structure of the argument; since it is constant relative to $\I$, having it will perform an affine shift of the connectivities we compute (in fact, making them increasingly connective). We can therefore get a lower bound by calculating in the case that $W=0$.

The connectivity of $\tilde X_n$ is at least the connectivity $(2n - 2d - 1)$ of the Thom space $\sJ(V,\bR^n)$.
Similarly, the connectivity of the structure map $S^m \wedge \tilde X_n \to \tilde X_{m+n}$ is at least $(n+m-2d)$.

Let $\bf n_{1}\to \bf n_{2}$ be a map in $\I$, and consider the following
factorization of the structure map in the B\"okstedt $\I-G$-space:
\[
\xymatrix{
  \Omega^{n_1\rho}\big(N_e^{C_k}(\tilde X_{n_1}) \sma S^{W}\big) \ar_{\Omega^{n_{1}\rho}E^{(n_{2}-n_{1})\rho}}[d] \ar[drr] & & \\
\Omega^{n_1\rho} \Omega^{\rho(n_2 - n_1)} \Sigma^{(n_2 - n_1)\rho}
\big(N_e^{C_k}(\tilde X_{n_1}) \sma S^W\big) \ar[rr]_-{\Omega^{n_{2}\rho}N_{e}^{C_k}\sigma} & & \Omega^{n_2\rho}\big(N_e^{C_k}(\tilde X_{n_2}) \sma S^W\big)
}
\]
where $E^{(n_{2}-n_{1})\rho}$ is the equivariant suspension and where $\sigma$ is the structure map in $X$: $\Sigma^{(n_{2}-n_{1})}X_{n_{1}}\to X_{n_{2}}$.

The structure map in the $\I-G$-space is at least as connected as 
\[
\min \big( \conn(\Omega^{n_{1}\rho}E^{(n_{2}-n_{1})\rho}),\conn(\Omega^{n_{2}\rho}N_{e}^{C_k}\sigma) \big),
\]
and we analyze each piece individually.

We first analyze the connectivity of $\Omega^{n_{1}\rho}E^{(n_{2}-n_{1})\rho}$. Applying Proposition~\ref{prop:freudenthal} and Lemma~\ref{lem:mappingspaceconnectivity}, we get a lower bound for the connectivity of the map $\Omega^{n_{1}\rho}E^{(n_{2}-n_{1})\rho}$:    
\begin{multline*}
\nu_{1}(H)=\conn\big((\Omega^{n_{1}\rho} E^{(n_{2}-n_{1})\rho})^{H}\big)\\
\geq \min_{K\subset H}\big(2\tfrac{k}{|K|}\conn(\tilde X_{n_{1}})-\tfrac{k}{|K|}n_{1}\big) \geq \tfrac{k}{|H|}\big(3n_1 - 4 d - 2\big).
\end{multline*}

Similarly, applying the norm $N_{e}^{C_{k}}$ to the structure map
$\Sigma^{n_{2}-n_{1}}X_{n_{1}}\to X_{n_{2}}$ yields our map 
\[
N_{e}^{C_{k}}\sigma\colon \Sigma^{(n_{2}-n_{1})\rho}N_{e}^{C_k}(X_{n_{1}})\cong N_{e}^{C_k}(\Sigma^{n_{2}-n_{1}}X_{n_{1}})\to N_{e}^{C_k}(X_{n_{2}}).
\]

By Lemma~\ref{lem:normconnectivity}, this norm produces a map that, on $H$-fixed points, has connectivity at least
\begin{multline*}
(n_1+n_2-2d) + 
(\tfrac{k}{|H|}-1)\big((2 n_1 - 2d - 1)+(n_2 - n_1)+1\big)\\
=\tfrac{k}{|H|}(n_1 +n_2- 2d).
\end{multline*}

By Lemma~\ref{lem:mappingspaceconnectivity} we can also estimate the connectivity of the $(n_{2}\rho)$-fold loops of this map, getting
\begin{multline*}
\nu_{2}(H)=\conn\big((\Omega^{n_{2}\rho}N_{e}^{C_{k}}\sigma)^{H}\big)\\
\geq\min_{K\subset H}\big(\tfrac{k}{|K|}(n_1+n_2-2d)-\tfrac{k}{|K|}n_{2}\big)= \tfrac{k}{|H|}(n_1 - 2d). 
\end{multline*}

The factorization of the structure map above shows us that the
structure map in the $\I-G$-space has connectivity at least
\[
\nu(H)=\min \big( \nu_{1}(H),\nu_{2}(H)\big) \geq \min \big(3n_1-4d-2, n_1-2d\big) \geq n_1-4d-2.
\]
(and usually a good bit more so).

Since this term goes to infinity,
Lemma~\ref{lem:Telescope} implies that $D_{1}$ is an equivariant
equivalence, as required.
\end{proof}


\section{Equivariance}\label{sec:diagonal}

The maps $D_1$ and $D_2$ allow us reduce from diagrams over $\I$ with
a trivial action to diagrams over $\J$ with a trivial
action of $C_k$.  But one of the interesting aspects of the equivariant
structure of the B\"okstedt smash product is that $C_k$ acts on the
diagram the homotopy colimit is indexed over.  In this section, we
show that the ``diagonal'' map $\Delta$ induces an equivariant
equivalence. This allows us to reduce to the case of homotopy colimits
on which the group acts trivially, as analyzed in the previous section.

The map $\Delta = \Delta_{1}$ is really part of a compatible family of maps, $\Delta_{i}$,
indexed on the divisor poset of $k$. Analyzing the equivariant
homotopy type of the B\"okstedt construction is facilitated by a refinement of Hesselholt-Madsen's determination of the fixed points of the B\"okstedt construction.
%

Suppose $k=d\cdot s$. Then we have a $C_k$-equivariant diagonal map
\begin{equation}
  \label{eq:diagonalmap}
\xymatrix{
\hocolim\limits_{\I^d} \Omega^{s(n_1+\ldots+n_d)}\left(N_{C_d}^{C_k}\left(\bigwedge\limits_{i=1}^d \tilde{X}_{n_i}\right) \sma S^W\right) \ar[d]^{\Delta_{d}} \\
\hocolim\limits_{\I^k} \Omega^{n_1+\ldots+n_k}\left(\bigwedge\limits_{i=1}^k \tilde{X}_{n_i} \sma S^W\right). \
} 
\end{equation}
The following proposition describes a key property of this generalized
diagonal map. In this, we fix our model for the homotopy colimit as the two-sided bar construction.

\begin{proposition}\label{prop:hmondrugs}
Suppose $k=d\cdot s$.  Then the map of equation~(\ref{eq:diagonalmap})
induces a homeomorphism on passage to $C_r$-fixed points for
$C_s \subseteq C_r \subseteq C_k$.
\end{proposition}

\begin{proof}
Our proof is a slight elaboration of that of Hesselholt-Madsen for the case $s=k$ \cite[\S2.4]{HeMa97}. The model for the homotopy colimit identifies the simplicial space realizing the homotopy colimit over $\I^d$ with a (nicely embedded) subsimplicial space of the homotopy colimit for $\I^k$, and as $s$ varies, these spaces are appropriately nested. Passage to fixed points commutes with geometric realization and with the products making up the two-sided bar construction for the homotopy colimit. This is what allows Hesselholt-Madsen to observe that the map $\Delta_1$ induces a homeomorphism on $C_k$-fixed points. Since $\Delta_1$ factors through $\Delta_d$ for any $d$, and since the inclusion is inducing a homeomorphism on fixed points in this case, we learn that $\Delta_d$ also induces an homeomorphism of $C_k$-fixed points (since these are really statements about fixed subspaces of a big ambient space). Downward induction on the group provides the rest of the result.
\end{proof}

We use this proposition to show the following:
\begin{theorem}
The map
\[
\xymatrix{
\hocolim\limits_{\I} \Omega^{kn}(N_e^{C_k} \tilde{X}_{n} \sma
S^W) \ar[r]^-{\Delta_1} & \hocolim\limits_{\I^k} \Omega^{n_1 + \ldots +
n_k}(\tilde{X}_{n_1} \sma \ldots \sma \tilde{X}_{n_k} \sma S^W)  
}
\]
induces a stable equivalence of equivariant orthogonal spectra.
\end{theorem}

\begin{proof}
Once again, we show that the map is a weak equivalence of $C_k$-spaces
for each $W$; the maps therefore assemble to form a level equivalence
and hence a stable equivalence of equivariant orthogonal spectra.

We need to show that the map is an equivalence on $C_s$-fixed points
for each $s \mid k$. Consider the composite 
\[ 
\xymatrix{
\hocolim\limits_{\I} \Omega^{kn}(X_n \sma \ldots \sma X_n \sma S^W) \ar[d]  \\
\hocolim\limits_{\I^d} \Omega^{s(n_1+\ldots+n_d)}((X_{n_1} \sma \ldots \sma X_{n_d})^{\sma s} \sma S^W) \ar[d] \\
\hocolim\limits_{\I^k} \Omega^{n_1+\ldots+n_k}(X_{n_1} \sma \ldots \sma X_{n_k} \sma S^W),
} 
\]
where $k = d \cdot s$ as before.
The second map is an equivalence on $C_s$-fixed points by
Proposition~\ref{prop:hmondrugs}, and we claim that the first map is
an equivalence on $C_s$-fixed points as well.  Consider the diagram
\[ 
\xymatrix{
\hocolim\limits_{\J} \Omega^{kn}(\bigwedge\limits_k X_n \sma S^W) \ar[r] \ar[d] & \hocolim\limits_{\J^d} \Omega^{s(n_1+\ldots+n_d)}((\bigwedge\limits_{i=1}^d X_{n_i})^{\sma s} \sma S^W) \ar[d] \\
\hocolim\limits_{\I} \Omega^{kn}(\bigwedge\limits_k X_n \sma S^W)
\ar[r] & \hocolim\limits_{\I^d} \Omega^{s(n_1+\ldots+n_d)}((\bigwedge\limits_{i=1}^d X_{n_i})^{\sma s} \sma S^W)
} 
\]
The point is that $C_s$ acts trivially on all the indexing categories
in the diagram, and the diagonal map $\J \to \J^d$ is homotopy
cofinal.  Since the passage to fixed points commutes with these
homotopy colimits, we conclude that the top horizontal map is a
$C_s$-equivalence.

The left hand side vertical map is an equivariant equivalence by 
Section~\ref{sec:diagrams}; the right hand side vertical map is an
equivariant equivalence for the same reason, using Fubini's theorem
for homotopy colimits (e.g., see~\cite[24.9]{ChacholskiScherer}). 
\end{proof}

This completes the proof of Theorem \ref{Main}. We have shown that the
vertical maps in the main comparison diagram (Equation~\ref{eq:main})
from Section~\ref{sec:intro} do indeed assemble to weak equivalences
of equivariant orthogonal spectra.  We have also identified the
orthogonal $C_k$-spectrum at the bottom of the comparison diagram with
the spectrum $N^{C_k}_e(X)$.

\appendix
\section{On Equivariant Homotopy Colimits and Cofinality}

Our proof of the equivalence of the map $\Delta_{1}$ was a direct argument. A more categorical approach uses an equivariant notion of homotopy cofinality.  For our purposes, we write $G$-category to denote an internal category in $G$-sets. A $G$-functor from a $G$-category $\cC$ to $\Top_G$ is a
functor which commutes with the $G$-actions.  The fixed points of a
$G$-category $\cC$ is an ordinary category, and the restriction of a
$G$-functor $F$ to $\cC^G$ is an ordinary functor into $G$-spaces. We
stress that the $G$-action is {\em{not}} assumed to consist of morphisms in
the category. This is essential to ensure that the resulting homotopy
colimits have a $G$-action. The prototypical example is $G=C_k$,
$\cC=\I^k$, and our functor is B\"okstedt's construction. 

The following lemma is essentially obvious (though we remark that it
is only true in spaces), relating the fixed points of a homotopy
colimit of a $G$-functor to the homotopy colimit of the fixed points. 

\begin{lemma}
If $\cC$ is a $G$-category and $F\colon \cC\to\Top_G$ is a
$G$-functor, then the natural map of $G$-spaces
\[
\hocolim_{\cC^G} F(c)\to\hocolim_{\cC}F(c)
\]
is an equivalence on $G$-fixed points, and the map of spaces
\[
\hocolim_{\cC^G}F(c)^G\to\big(\hocolim_{\cC}F(c)\big)^G
\]
is an equivalence.
\end{lemma}

\begin{proof}
We choose as a model for $\hocolim$ the two-sided bar construction
$B(*,\cC,F)$. This is the geometric realization of a simplicial space,
and since $G$-fixed point functor commutes with geometric realization, we
need only show that this equality is true for the $r$-simplices of the
bar construction for each $r$. Since we are in spaces, the fixed
point space of a product is the product of the fixed point spaces,
and we conclude
\[
(\hocolim_{\cC}F(c))^G\cong \hocolim_{\cC^G}F(c)^G.
\]
In fact, since we have chosen a particular model, these spaces are
actually homeomorphic.  The right most term is visibly
$(\hocolim_{\cC^G}F(c))^G$. 
\end{proof}

The following trivial generalization of the preceding result is
surprisingly useful:

\begin{corollary}\label{cor:Subcats}
Suppose that $\cC$ is a $G$-category and $\D$ is a $G$-subcategory such that $\cC^G\subset \D$. Then for any $G$-functor $F$ from $\cC$ to $\Top_G$, we have an equivalence of fixed points
\[
\big(\hocolim_{\D}F(d)\big)^G\to\big(\hocolim_{\cC}F(c)\big)^G.
\]
\end{corollary}
\begin{proof}
The assumptions on $\cC$ and $\D$ ensure that $\cC^G=\D^G$. The result follows from identifying both with this common colimit.
\end{proof}

To complete the proof of the equivalence of $\Delta_{1}$, we can directly appeal to an equivariant form of Quillen's Theorem A \cite[Theorem 3.10]{vflor}.

\begin{theorem}
Let $F\colon\cC\to\D$ be a functor such that for every $d\in \D$, the
over-category $F\downarrow d$ is $Stab(d)$-equivariantly contractible. Then $F$ induces an equivalence on homotopy colimits.
\end{theorem}

In our context, this version of Quillen's Theorem A shows that the natural map $\J\to \J^{k}$ is equivariantly cofinal. By the B\"okstedt Telescope Lemma and the commutativity of the obvious diagram of homotopy colimits, we conclude that $\Delta_{1}$ is an equivariant equivalence.

\end{document}